\documentclass{article}

\usepackage{amsthm}
\usepackage{amsmath}
\usepackage{amssymb}
\usepackage{graphicx}
\usepackage{times}

\renewenvironment{abstract}
 {\small
  \begin{center}
  \bfseries \abstractname\vspace{-.5em}\vspace{0pt}
  \end{center}
  \list{}{%
    \setlength{\leftmargin}{10mm}
    \setlength{\rightmargin}{\leftmargin}%
  }%
  \item\relax}
 {\endlist}

\newtheorem{theorem}{Theorem}
\newtheorem{definition}[theorem]{Definition}
\newtheorem{proposition}[theorem]{Proposition}

\newtheorem{corollary}[theorem]{Corollary}
\newtheorem{remark}[theorem]{Remark}

\title{Computing Your Ideal Haircut Routine}
\author{Blake Pehrson}
\date{October 2024}

\begin{document}

\maketitle

\begin{abstract}
    We introduce stochastically resetting deterministic processes - the simplest subclass of general resetting stochastic processes - finding them to be repackaged renewal processes. In particular, we consider the stochastically resetting deterministic process undergoing linear growth of rate $1$ subject to Poissonian resetting and deduce its marginal, expectation, and variance. Then, using these results, we construct a stochastic model for hair growth subject to regular haircuts and use this model to prescribe the reader's ideal haircut routine.
\end{abstract}

\section{Introduction}

    \subsection{Motivation}
    I invite the reader to participate in the following: Consider a single hair follicle and the hairs that have naturally grown from it - unfettered by any haircuts. Imagine you have been observing this tiny system for a long while and have drawn a graph, plotting hair length against time. What does this graph look like? What are the graph's defining characteristics? What significant events occurred during observation, and what happened between them? You may be imagining a graph like the one shown in Figure \ref{fig:timing_process_sample_path}.

    Reduced to a base mathematical level, ignoring haircuts for the time being, a hair grows linearly, until it falls out - only to be replaced by another, repeating the linear growth of the first. The cycle continues with successive hairs falling out at varying intervals. Explicitly, between events occurring stochastically in time, we observe a repeating deterministic system. We call such a process a ``stochastically resetting deterministic process".

    \subsection{Resetting Stochastic Processes}
    
    ``Stochastically resetting deterministic processes" are perhaps the simplest class of the more general ``resetting stochastic processes", whereby a stochastic process $(S_t)_{t\in T}$ evolves freely in time until being reset to some predetermined starting point $s_0$. Once reset, the process continues to evolve freely. The cycle continues, with times between resets being independent and identically distributed. See \cite[p.36]{evans2020stochastic} for a more formal definition. Resetting stochastic processes are a relatively new discipline of study, first being studied in the late 1970s. Since then, most research regarding resetting stochastic processes has been done in the field of mathematical physics towards creating algorithms and modeling mathematical systems. See \cite{montero2017continuous} for a historical overview.

    The standard jumping off point into resetting stochastic processes is to consider Brownian motion under Poissonian resetting, which means that the process evolves as a particle undergoing diffusion with a constant chance of being reset at each infinitesimal moment \cite{gupta2022stochastic}. This process illustrates the primary characteristics of interest when researching resetting stochastic processes. Specifically, Brownian motion under Poissonian resetting permits a stationary distribution, where Brownian motion does not. Moreover, the expected time to reach a target point becomes finite upon imposing resetting \cite[p.5-6]{gupta2022stochastic}, and is minimized by an optimal resetting rate. Together, these observations - which extend beyond this one example - are of particular interest when considered in relation to search algorithms, suggesting an optimal non-zero resetting rate will improve the odds of finding a target \cite[p.2]{evans2020stochastic}.

    We will see that while stochastically resetting deterministic processes do permit stationary distributions, they do not have applications to search algorithms, as a deterministic process would either always reach a target point in a predetermined amount of time or never reach a target point. However, stochastically resetting deterministic processes have alternative use in mathematical modeling, as explained with the example of a growing hair. Other motivating examples could include the bacteria buildup on a surface which is randomly cleaned or a candy jar which is randomly refilled.

    \subsection{Overview}
    
    In this paper, we give a brief introduction to stochastically resetting deterministic processes, illustrate their key characteristics, and demonstrate their use in mathematical modeling.
    
    In section \ref{sec:theory}, we formalize what class of stochastic processes we are referring to when we say ``a stochastically resetting deterministic process". Moreover, we prove some general properties and compute the marginal of a canonical example.

    Next, using the theory we develop in section \ref{sec:theory}, we construct a stochastic model for hair growth and haircuts in section \ref{sec:application to the modeling of hair growth}. Using this model, we predict the sample mean and variance of a head of hair.

    Then, in section \ref{sec:quantifying unkempt hair} we suggest a standardized scale on which to describe the appearance of a head of hair and conclude with a novel exercise in which the reader may compute their ideal haircut routine.

    Finally, section 5 concludes with three other exercises for the reader, where this theory can be applied.
    
\section{Theory}\label{sec:theory}

    In this section we begin by constructing the most basic class of stochastically resetting deterministic processes, wherein - between resets - the process grows linearly from $0$ with rate $1$. We then compute the marginal law in the case of Poissonian resetting. Next, generalizing to a broader class of deterministic processes, we examine how the Markov property is related to these processes. Lastly, we consider the behavior of these processes continuing on from a first hitting time and compute the marginal law in our carried example.

    \subsection{Timing Processes}
    
    Before we can address a stochastically resetting deterministic process, we clarify what we mean by a deterministic process. A deterministic process in time is one whose state is exactly determined by the current point in time. Explicitly:
    
    \begin{definition}[Deterministic Process]\label{def: deterministic process}\ \\
    For a function $d:[0,\infty)\xrightarrow{}\mathbb{R}$, the process $(d(t))_{t\geq0}$ is called a \textbf{deterministic process}.
    \end{definition}

    To gain insight into stochastically resetting deterministic processes, and to formulate a method of construction, we return to the motivating example of a growing hair. Assuming that successive hairs grow at a constant rate, an individual hair's growth proceeds deterministically along some linear line starting at zero. Therefore, by knowing when hairs fall out, we can fill in the gaps between these times and describe the system in its entirety. More generally, for any repeating deterministic process, the entire process is determined precisely by the renewal times - when the process resets to some established starting point. Using this observation, we need only model the renewal times and then build the process atop them.
    
    \begin{definition}[Timing process]
    \label{def:timing process}\ \\
        Let $(X_i)_{i\in\mathbb{N}}$ be a sequence of positive i.i.d. r.v.s defined on probability space $(\Omega,\mathcal{F},\mathbb{P})$. For $n\in \mathbb{N}_0$ define r.v.s 
        \[
        Y_n=\sum_{i=1}^{n}X_i,
        \]
        where $Y_0=0$ in particular. Then, define the stochastic process $(Z_t)_{t\geq0}$ by
        \[
        Z_t=inf\{t-Y_n|n\in\mathbb{N}_0,t-Y_n\geq0\}
        \]
        noting that for all $t\geq0$ we have $t=t-0=t-Y_0$. Therefore, the prescribed set is non-empty. Moreover, by definition, the set is bounded below, so the infimum exists and is non-negative. Linear combinations and infimums preserve measurability, so the defined functions are random variables.
        Call $(Z_t)_{t\geq0}$ a \textbf{timing process} with \textbf{renewal times} $(Y_n)_{n\in\mathbb{N}_0}$ and \textbf{inter-renewal times} $(X_i)_{i\in\mathbb{N}}$. Alternatively, we may say $(Z_t)_{t\geq0}$ is the timing process induced by inter-renewal time $X$.
    \end{definition}
    \begin{remark}\label{rem:recording time}
        Observe that $(Z_t)_{t\geq0}$ is recording the time since the most recent renewal time. Therefore, the process could be interpreted as the readout of a stopwatch whose reset button is pressed at each renewal time. Hence the name ``timing process".
    \end{remark}
    \begin{remark}\label{rem:i didn't do anything special}
        As observed in the motivation for Definition \ref{def:timing process}, the information in this process is equivalent to the information contained in the sequence of i.i.d. inter-renewal times $(X_n)_{n\in\mathbb{N}}$. The reader may recognize this to be the same information upon which renewal processes are constructed (and in the case of i.i.d. exponential inter-renewal times, the Poisson process). Therefore, a timing process is essentially a repacked renewal process. Indeed, this random variable is sometimes called the lifetime of a renewal process and is related to the more well known residual lifetime.
    \end{remark}

    \begin{figure}
        \centering
        \includegraphics[width=2.66in]{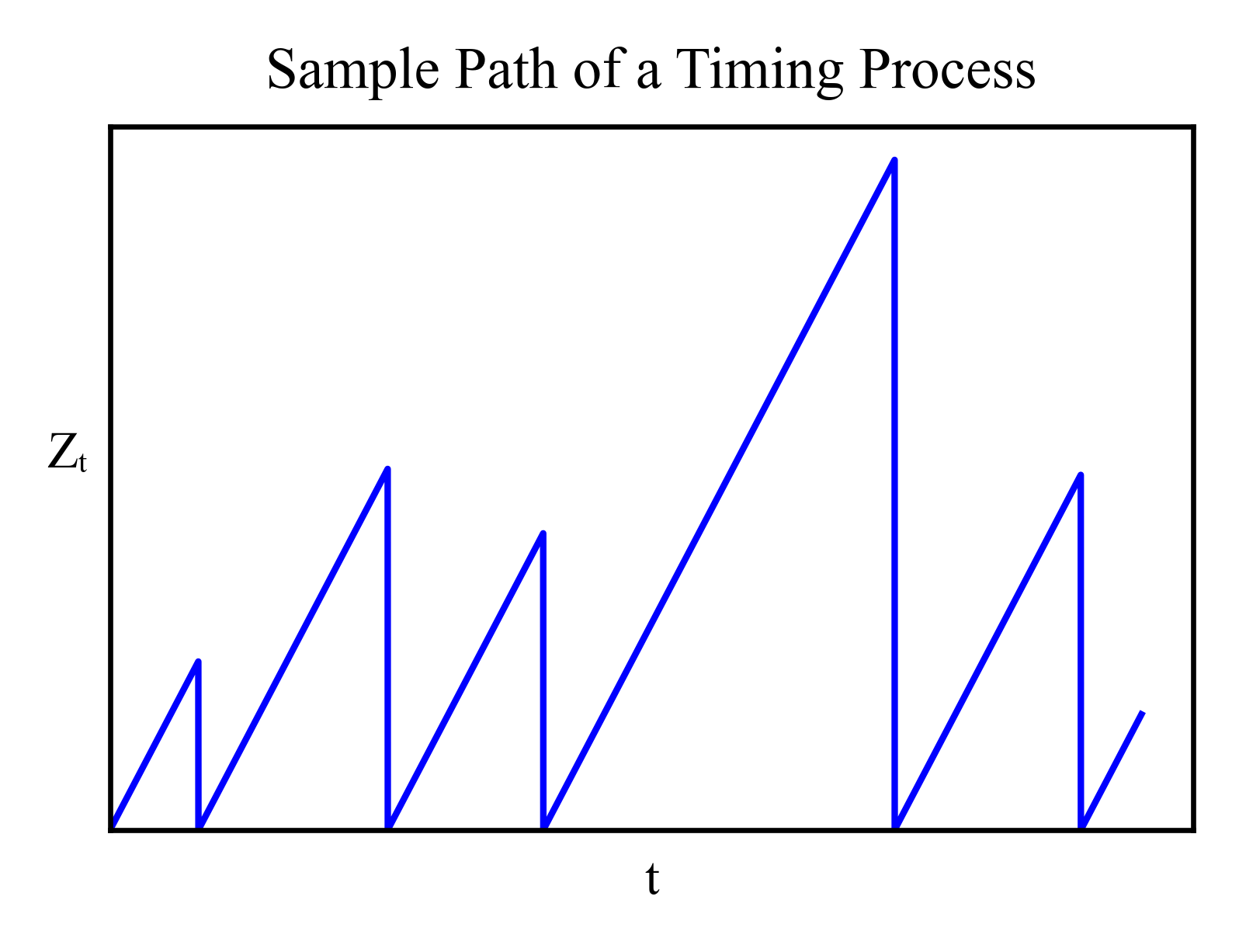}
        \caption{A sample path of a timing process with exponential inter-renewal times.}
        \label{fig:timing_process_sample_path}
    \end{figure}

    Definition \ref{def:timing process} only considers linear growth of rate 1. However, because deterministic processes are determined exactly by time, and by Remark \ref{rem:recording time}, we will later be able to adapt this process to most deterministic processes by way of functional composition. Before we generalize this idea, we consider timing processes in greater detail.

    We know present a standard result from renewal theory, considered from the perspective of resetting processes:
    
    The exponential distribution is commonly used to model random waiting times. Due to its desirable properties and relations to other distributions, we are able to explicitly compute the marginal of a timing process induced by exponential inter-renewal times below, demonstrating the primary strategies for dealing with timing processes. In the literature, this form of resetting is referred to as ``Poissonian" resetting - so named because the number of resets occurring within a fixed time frame is Poisson distributed. Considered at the infinitesimal level, Poissonian resetting occurs when the process has a constant chance of being reset at each moment.

    \begin{proposition}[Marginal of exp$(\lambda)$ timing process]
    \label{prop:marginal of exp timing process}\ \\
        Let $\lambda>0$ be given and suppose $(Z_t)_{t\geq0}$ is a timing process induced by the inter-renewal time exp$(\lambda)$. Then, for any $t\geq0$,
        \[
        \mathbb{P}[Z_t\leq z]=
        \begin{cases} 
              0 & z\in(-\infty,0) \\
              1-e^{-\lambda z} & z\in[0,t) \\
              1 & z\in[t,\infty)
        \end{cases}.
        \]
    \end{proposition}
    \begin{proof}
        Let $(Y_n)_{n\in\mathbb{N}_0}$ and $(X_i)_{i\in\mathbb{N}}$ be the renewal and inter-renewal times of $(Z_t)_{t\geq0}$ respectively.
        \\
        Let $t\geq0$ be given. By the definition of a timing process, $Z_t\geq0$. Moreover, the largest value $Z_t$ could achieve is $t$, precisely when $Y_1>t$ (the first renewal time occurs after time $t$). Therefore $0\leq Z_t\leq t$, and the first and last cases of the function are trivial. Hence, assume $z\in[0,t)$.

        Consider figure \ref{fig:timing_process_sample_path} showing a sample path of a timing process. In particular, note how between renewal times the process grows linearly from $0$. We can make stronger statements within these intervals, so our strategy for the rest of the proof will be to condition upon which pair of successive renewal times bound our time $t$.
        
        By the definition of a timing process, $(X_i)_{i\in\mathbb{N}}$ is a sequence of positive random variables. Therefore, the sequence of partial sums $(Y_n)_{n\in\mathbb{N}_0}$ is strictly increasing and unbounded almost surely (with probability $1$), and we have
        \[
        \mathbb{P}[Z_t>z]=\mathbb{P}[Z_t> z,0=Y_0<Y_1<... , (Y_n)_{n\in\mathbb{N}_0} \text{unbounded}].
        \]
        Then, there must exist a unique $n\in\mathbb{N}_0$ such that $Y_n\leq t<Y_{n+1}$. This allows us to partition
        \[
        \mathbb{P}[Z_t>z]=\sum_{n=0}^{\infty}\mathbb{P}[Z_t>z, Y_n\leq t<Y_{n+1}].
        \]
        Now, using that $z\geq0$ and the fact that for $t\in[Y_n,Y_{n+1})$, $(Z_t)_{t\geq0}$ grows linearly from $0$, we can compute $Z_t=t-Y_n$ and say that
        \[
        Z_t>z, Y_n\leq t<Y_{n+1} \iff t-Y_n>z, Y_n\leq t<Y_{n+1} \iff Y_n+z<t<Y_{n+1}.
        \]
        Recalling that $Y_{n+1}=Y_n+X_{n+1}$ we deduce that
        \[
        Y_n+z<t<Y_{n+1} \iff (X_{n+1},Y_n)\in A
        \]
        where $A=\{(x,y)\in\mathbb{R}^2|t-x<y<t-z\}$.
        Hence, substituting into our above sum, we find that
        \[
        \mathbb{P}[Z_t>z]=\sum_{n=0}^{\infty}\mathbb{P}[(X_{n+1},Y_n)\in A].
        \]
        Using the Independence of $Y_n$ and $X_{n+1}$, $X_{n+1}\sim exp(\lambda)$, and the fact that the sum $Y_n$ of $n$ i.i.d. $exp(\lambda)$ random variables is $\Gamma(n,\lambda)$ distributed, we compute
        \[
        \mathbb{P}[(X_{n+1},Y_n)\in A]=\frac{(\lambda(t-z))^n e^{-t\lambda}}{n!}.
        \]
        Finally, substituting into our sum, we see that
        \[
        \mathbb{P}[Z_t>z]=e^{-t\lambda}\sum_{n=0}^{\infty}\frac{(\lambda(t-z))^n e^{-\lambda(t-z)}}{n!}=e^{-t\lambda},
        \]
        by recognizing the terms of the sum as the PMF of a Poisson distribution with parameter $\lambda(t-z)$. Taking the probability of the complement then yields the required result.
    \end{proof}
    \begin{remark}\label{rem:mixed rv}
        This CDF describes a mixed random variable. In particular, $Z_t|Z_t<t$ is continuous and $Z_t|Z_t\geq t$ is discrete. In fact, $\mathbb{P}[Z_t=z]\neq0 \iff z=t$. At the start of the proof we noted that $Z_t=t \iff Y_1>t$. Computing this probability we indeed find that $\mathbb{P}[Z_t=t]=e^{-t\lambda}=\mathbb{P}[Y_1>t]$. Then, so long as the first renewal time has arrived before or at time $t$, we are in the realm of continuous random variables.
    \end{remark}
    \begin{remark}\label{rem:converge to X}
        Note that the CDF of $Z_t$ is the CDF of $X_1$ with a discontinuity jump to $1$ at time $t$. As $t\xrightarrow{}\infty$, this dividing line is pushed further along the $x$-axis. In particular this means that
        \[
        Z_t\xrightarrow{d}X_1 \text{ as } t\xrightarrow{}\infty.
        \]
        This does not hold true for a general timing process when $X_1$ is not exponentially distributed (Consider the case when $X_1 \sim U[1,2]$).
    \end{remark}
    \begin{remark}\label{rem:minimum charactarization}
        As in Remark \ref{rem:converge to X}, observe that the CDF of $Z_t$ is the CDF of an $exp(\lambda)$ random variable, but with a discontinuity jump at time $t$. Using this observation, it is not difficult to see that
        \[
        Z_t \sim min\{exp(\lambda),t\}.
        \]
    \end{remark}

    \begin{figure}
        \centering
        \includegraphics[width=2.6916in]{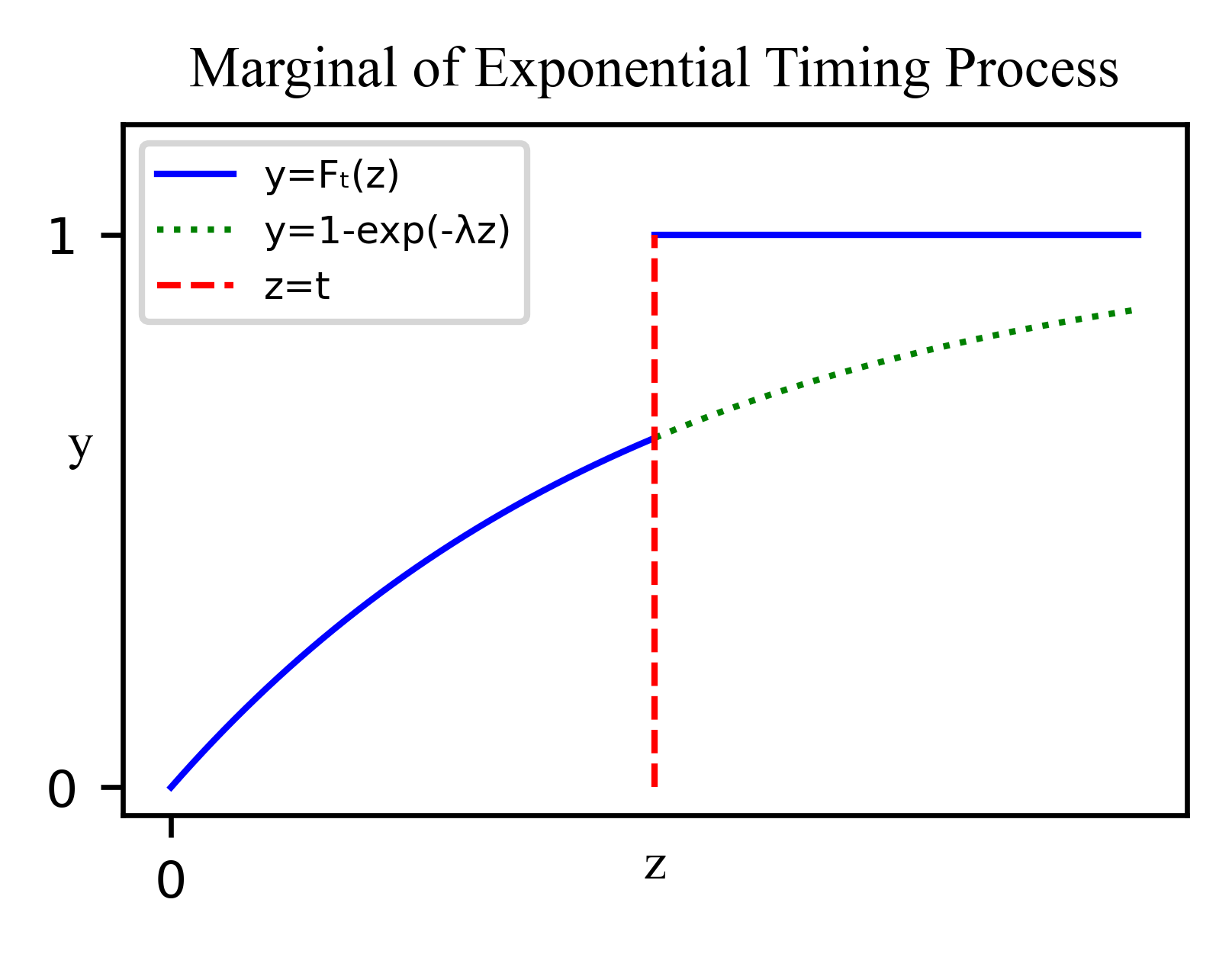}
        \caption{The CDF of $Z_2$ where $(Z_t)_{t\geq 0}$ is a timing process induced by inter-renewal time $exp(0.5)$.} 
        \label{fig:timing_process_0_marginal}
    \end{figure}

    \subsection{Resetting Processes and the Markov Property}
    
    We now generalize to a broader class of resetting deterministic processes.
    
    \begin{definition}[Resetting process]\label{def:resetting process}\ \\
        Let $(Z_t)_{t\geq0}$ be a timing process and $d:[0,\infty)\xrightarrow{}\mathbb{R}$ a Borel function. Call the stochastic process $(d(Z_t))_{t\geq0}$ a \textbf{resetting process}. We may say $(d(Z_t))_{t\geq0}$ is induced by $d$.
    \end{definition}

    Now, for a Borel function $d:[0,\infty)\xrightarrow{}\mathbb{R}$ and distribution function $F:\mathbb{R}\xrightarrow{}[0,1]$ with positive support, we can construct a stochastically resetting deterministic process with maximal sample path $d$ and i.i.d. inter-renewal times whose individual laws are given by $F$. Figure \ref{fig:sin_resetting_process_sample_path} shows a sample path of a resetting process induced by the function $sin:[0,\infty)\xrightarrow{}\mathbb{R}$.
    
    \begin{figure}
        \centering
        \includegraphics[width=2.94in]{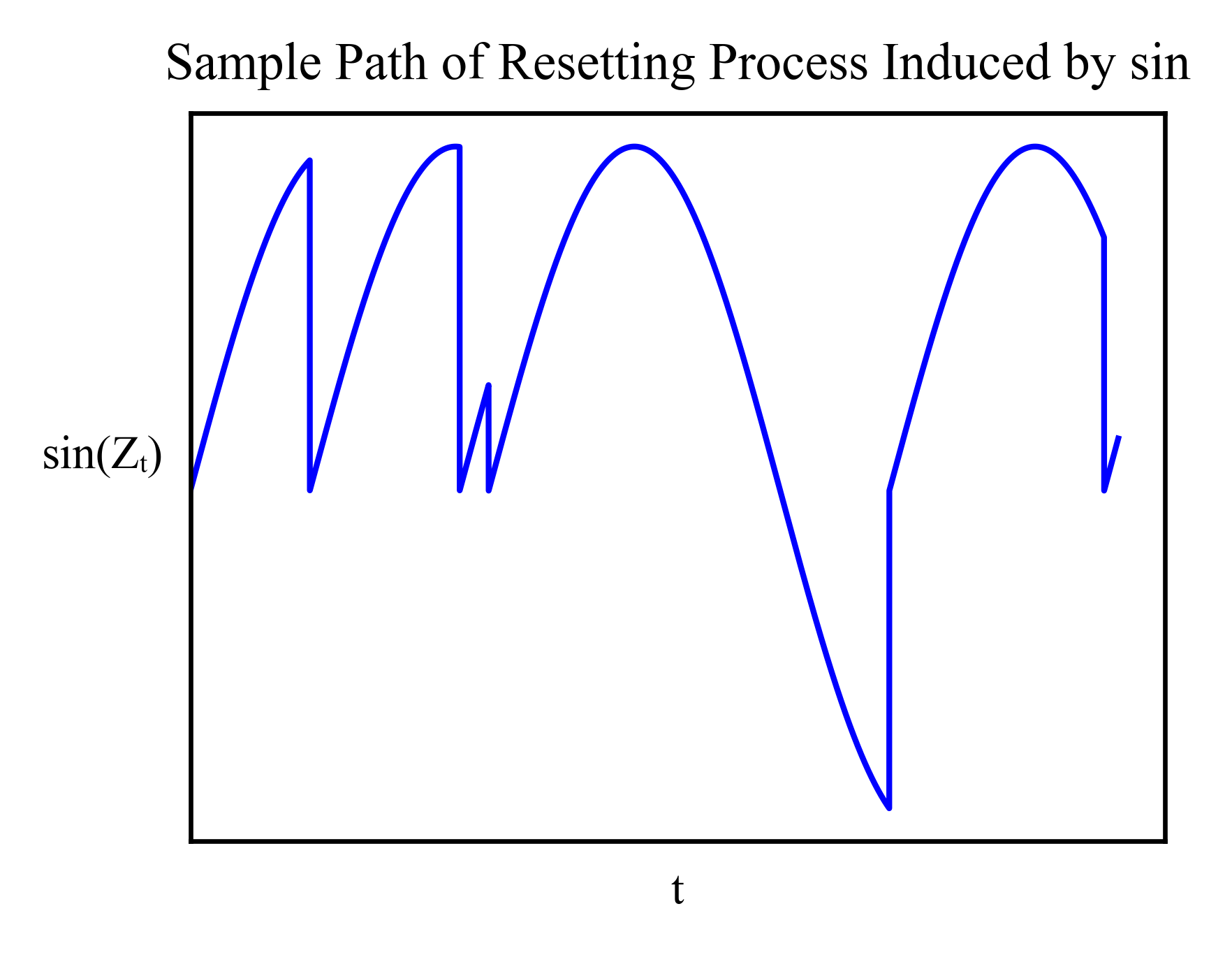}
        \caption{A sample path of a resetting process with exponential inter-renewal times and induced by the function $sin:[0,\infty)\xrightarrow{}\mathbb{R}$.}
        \label{fig:sin_resetting_process_sample_path}
    \end{figure}

    We now consider what we can predict about a stochastically resetting deterministic process by observing its current state. We will see that the Markov property is interlinked with this idea.
    
    Given the sample path of a resetting process up to some present state, the immediate future is almost surely deterministic. A stronger condition would impose that given only the present, the immediate future is almost surely deterministic. In other words, the best prediction of the future depends only on the present state, which is exactly the Markov property. In the following proposition and corollary we explain a sufficient condition for a resetting process to be Markov.
    
    \begin{proposition}[Timing processes are Markov]
    \label{prop:timing processes are markov}\ \\
    Let $(Z_t)_{t\geq0}$ be a timing process. Then $(Z_t)_{t\geq0}$ is Markov.
    \end{proposition}
    \begin{proof}
        Let $(Z_t)_{t\geq0}$ have renewal times $(Y_n)_{n\in\mathbb{N}_0}$ and inter-renewal times $(X_i)_{i\in\mathbb{N}}$. Further assume $ s \geq0$.
        
        At time $s$, we can determine the most recent renewal time $Y_m= s -Z_ s $ so that $Y_m\leq s <Y_{m+1}$. Then observe that given $Z_ s $, the past $(Z_t)_{t< s }$ is determined precisely by $(Y_n)_{n=0}^{m-1}$ and hence by $(X_i)_{i=1}^{m-1}$.

        Next, observe that the future $(Z_t)_{t>s}$ is determined by $(Y_n)_{m\geq n}$. For $n\in\mathbb{N}$, we have that $Y_{m+n}=Y_m+\sum_{i=m+1}^{m+n}X_{i}$, and so for $t\geq  s $, $Z_t$ is determined precisely by $Y_m$ and $(X_i)_{i\geq m+1}$. But $Y_m$ is determined by $Z_ s $. Therefore, given $Z_ s $, the future $(Z_t)_{t> s }$ is determined by $(X_i)_{i\geq m+1}$.

        But $(X_i)_{i=1}^{m-1}$ and $(X_i)_{i\geq m+1}$ are independent by the definition of a timing process. Therefore, given the present, the future is independent of the past. This is equivalent to the Markov property.
    \end{proof}
    \begin{remark}
        In Remark \ref{rem:i didn't do anything special}, we explained that a timing process is a repackaged renewal process. It is interesting to note that renewal processes do not have the Markov property in general, owing to the inability to determine the most recent renewal time given the process's current state, arising as a result of a renewal process being constant between renewal times. However, provided that the renewal process has i.i.d exponential inter-arrival times, a renewal process reduces to a Poisson process which is Markov due to the memoryless property of exponential random variables.

        Similarly, stochastically resetting non-deterministic processes do not have the Markov property in general unless they are induced by i.i.d. exponential inter-renewal times. 
    \end{remark}

    \begin{corollary}[Resetting process induced by a given injective Borel functions is Markov]\label{cor:condition for markov}\ \\
    Let $(d(Z_t))_{t\geq0}$ be a resetting process induced by a given Borel function $d:[0,\infty)\xrightarrow{}\mathbb{R}$ and inter-renewal time $X$. If $d$ is injective then $(d(Z_t))_{t\geq0}$ is Markov.

    \end{corollary}
    \begin{proof}
        The proof is similar to that of Proposition \ref{prop:timing processes are markov}, excepting the fact that injectivity is sufficient to determine the most recent renewal time $Y_m$. To illustrate this, consider the following counter-example: take a resetting process induced by the $sin$ function (as shown in Figure \ref{fig:sin_resetting_process_sample_path}) and suppose are given $sin(Z_ s )=0$. By the periodicity of $sin$, the most recent renewal time could be any multiple of $\pi$. Given the past and present, the immediate future is deterministic, but with only the present, the process could increase or decrease.
    \end{proof}

    \subsection{Observing Ongoing Processes}

     With a better understanding of which resetting processes are Markov, we ask about the continued evolution of a process after hitting some target point.

    \begin{definition}[Timing process started at $\alpha\geq0$]
    \label{def:timing starting at alpha}\ \\
        Let $(Z_t)_{t\geq0}$ be a timing process induced by inter-renewal time $X$. For any $\alpha\geq0$ which is not an upper bound on the range of $X$, consider the first hitting time 
        \[
        \tau_\alpha=inf\{t\geq0|Z_t=\alpha\}
        \]
        and define the stochastic process 
        \[
        (Z_{t}^{\alpha})_{t\geq0}=(Z_{\tau_\alpha +t})_{t\geq0}.
        \]
        Call $(Z_{t}^{\alpha})_{t\geq0}$ a \textbf{timing process started at $\alpha$}.
    \end{definition}
    \begin{remark}\label{rem:natural behavior}
        With similar reasoning to what was shown in the beginning of the proof of Proposition \ref{prop:marginal of exp timing process}, if the range of inter-renewal time $X$ were bounded above by some $L>0$, then the timing process would almost surely be reset before reaching that point. Therefore, under natural behavior, a timing process can only take values which are not upper bounds on the range of its inter-renewal times.
    \end{remark}
    \begin{remark}\label{rem:immediate future markov}
        This definition is relevant precisely by the Markov property proven in Proposition \ref{prop:timing processes are markov}. One can extend this definition to all resetting processes, but without the Markov property, the immediate future at time $t=0$ would not necessarily be deterministic, which is what we are trying to examine.
    \end{remark}

    In many common stochastic processes such as Brownian motion, random walks, or Poisson processes, we have the notion of ``a process started at $\alpha$", where we ask about the continued behavior of the process after hitting $\alpha$. If the stochastic process $(S_t)_{t\in T}$ is one of the processes mentioned above, we may simply define $(S_t+\alpha)_{t\in T}$ to be the process started at $\alpha$. This follows from observing that the mentioned processes are all Lévy processes, meaning they have stationary and independent increments - less mathematically, they have no fixed reference point. If we took the same approach when considering a timing process started at $\alpha$, the process would no longer be reset to $0$, but to $\alpha$. Hence, we would not be describing the continued behavior of a timing process started at $\alpha$. Further analysis is required.
    
    We continue on from our example in Proposition \ref{prop:marginal of exp timing process}, deducing the marginal of the given timing process started at some $\alpha\geq0$.
    
    \begin{theorem}[Marginal of exp$(\lambda)$ timing process started at $\alpha\geq0$]
    \label{the:marginal of exp timing process started at a}\ \\
        Let $\lambda>0$ and $\alpha\geq0$ be given. Suppose $(Z_{t}^{\alpha})_{t\geq0}$ is a timing process induced by the inter-renewal time exp$(\lambda)$ and started at $\alpha$. Then, for any $t\geq 0$,
        \[
        \mathbb{P}[Z_{t}^{\alpha}\leq z]=
        \begin{cases} 
              0 & z\in(-\infty,0) \\
              1-e^{-\lambda z} & z\in[0,t) \\
              1-e^{-\lambda t} & z\in[t,t+\alpha) \\
              1 & z\in[t+\alpha,\infty) 
        \end{cases}.
        \]
    \end{theorem}
    \begin{proof}
        Before we begin, let us review the memoryless property of exponential random variables:

            Let $X\sim exp(\lambda)$ be a random variable defined on probability space $(\Omega,\mathcal{F},\mathbb{P})$. Then, for all $\beta,x\geq0$ we have
            \[
            \mathbb{P}[X>x+\beta|X>\beta]=\mathbb{P}[X>x]
            \]
            So that $(X|X>\beta) \sim X+\beta$.
            
        We continue with the proof. Let $(Z_t)_{t\geq0}$ be a timing process with i.i.d. $exp(\lambda)$ inter-renewal times $(X_i)_{i\in\mathbb{N}}$ and renewal times $(Y_n)_{n\in\mathbb{N}_0}$. Consider the first hitting time $\tau_\alpha \geq 0$. Because $(Y_n)_{n\in\mathbb{N}_0}$ is almost surely unbounded and strictly increasing with $Y_0=0$, we can bound $\tau_\alpha$ within a unique pair of renewal times $Y_m\leq \tau_\alpha <Y_{m+1}$. By the definition of $\tau_\alpha$, we have that $Z_{\tau_\alpha}=\alpha$. Then, because timing processes grow linearly from $0$ between successive renewal times, we have that $\tau_\alpha=Y_m+\alpha$. Therefore, we deduce that
        \[
        X_{m+1}=Y_{m+1}-Y_m>\alpha.
        \]
        By assumption, $X_{m+1} \sim exp(\lambda)$. Then, applying the memoryless property of exponential random variables and substituting for $Y_{m+1}$, we see that
        \[
        Y_{m+1}-\tau_\alpha \sim exp(\lambda).
        \]
        Then, by the independence of inter-renewal times, for $n\in\mathbb{N}$, $Y_{m+n}-\tau_\alpha$ is a sum of $n$ i.i.d. $exp(\lambda)$ random variables. Recognizing this to be identical to the construction of the original renewal times, we see that
        \[
        (Y_{m+n}-\tau_\alpha)_{n\in\mathbb{N}} \sim (Y_n)_{n\in\mathbb{N}}.
        \]
        Finally, by considering the values $(Z_t)_{t\geq 0}$ must take between these renewal times, the following construction produces an identically distributed random variable:
        \[
        Z_{t}^{\alpha}=
        \begin{cases} 
            \alpha+Z_t & t<Y_1 \\
            Z_t & Y_1\leq t \\
       \end{cases},
        \]
        where we simply shift the timing process up by $\alpha$ until we reach the first renewal time $Y_1$. From this construction, we need only condition on $Y_1$ to deduce the claimed marginal.
    \end{proof}
    \begin{remark}\label{rem:mixed2}
        Following on from Remarks \ref{rem:mixed rv} and \ref{rem:converge to X}, this marginal also describes mixed random variables. In particular, $Z_{t}^{\alpha}$ is discrete only for $Z_{t}^{\alpha}=t+\alpha$, which occurs when $Y_1>t$. Moreover, the starting state of a timing process does not change its limiting behavior, and we retain
        \[
        Z_{t}^{\alpha}\xrightarrow{d}exp(\lambda) \text{ as } t\xrightarrow{}\infty.
        \]
        It is also easy to verify that $exp(\lambda)$ is the stationary distribution of a timing process in the sense that if $X \sim exp(\lambda)$, then for all $t\geq0$
        \[
        Z_{t}^{X} \sim X.
        \]
    \end{remark}

    \begin{figure}
        \centering
        \includegraphics[width=2.8483in]{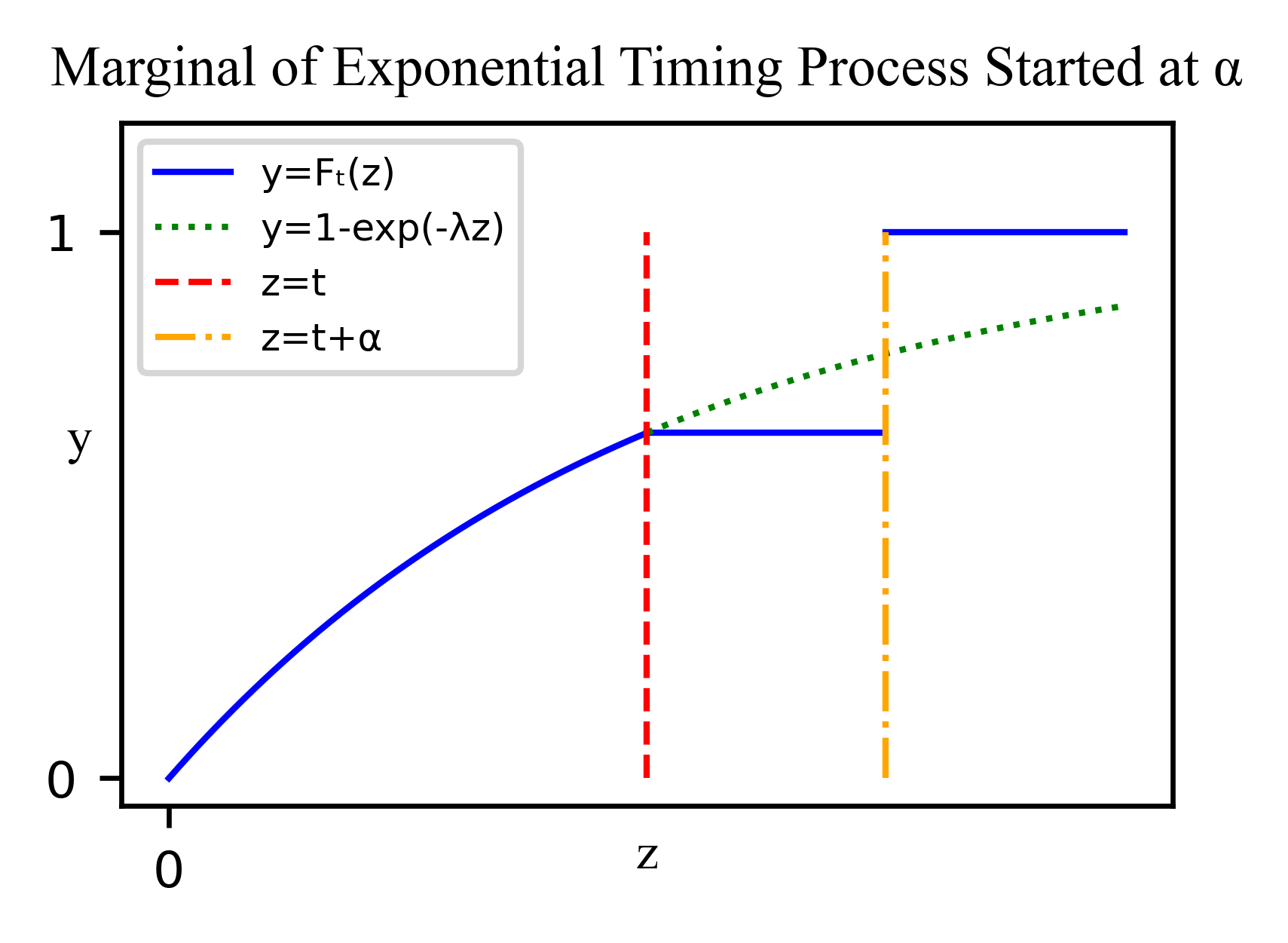}
        \caption{The CDF of $Z_{2}^{1}$ where $(Z_{t}^{1})_{t\geq0}$ is a timing process induced by inter-renewal time $exp(\frac{1}{2})$, and started at $1$.}
        \label{fig:enter-label}
    \end{figure}

    \begin{corollary}[Expectation and variance of exp$(\lambda)$ timing process started at $\alpha$]\label{cor:expectation and variance}\ \\
        Under the assumptions of Theorem \ref{the:marginal of exp timing process started at a}, we have that for all $t\geq0$,
        \[
        \mathbb{E}(Z_{t}^{\alpha})=\frac{1}{\lambda}(1-e^{-\lambda t}) +\alpha e^{-\lambda t}
        \]
        and
        \[
        Var(Z_{t}^{\alpha})=\frac{1}{\lambda^2}(1+e^{-t\lambda}(\alpha^2\lambda^2+2t\alpha\lambda^2-2\alpha\lambda-2t\lambda)+e^{-2t\lambda}(2\alpha\lambda-1-\alpha^2\lambda^2))
        \]
    \end{corollary}
    \begin{proof}
        As mentioned in Remark \ref{rem:mixed2}, $(Z_{t}^{\alpha})_{t\geq0}$ are mixed random variables. Therefore, to compute their expectation, we may condition upon their values according to whether or not a reset has yet occurred and use the tools of discrete and continuous probability separately. Explicitly,
        for $t,\alpha\geq0$, by the law of total expectation, we have that
        \[
        \mathbb{E}(Z_{t}^{\alpha})=\mathbb{E}(Z_{t}^{\alpha}|Z_{t}^{\alpha}=t+\alpha)\mathbb{P}[Z_{t}^{\alpha}=t+\alpha]+\mathbb{E}(Z_{t}^{\alpha}|Z_{t}^{\alpha}<t)\mathbb{P}[Z_{t}^{\alpha}<t]
        \]
        where $Z_{t}^{\alpha}|Z_{t}^{\alpha}=t+\alpha$ is discrete and $Z_{t}^{\alpha}|Z_{t}^{\alpha}<t$ is continuous. The claimed expectation follows. We can similarly compute $\mathbb{E}((Z_{t}^{\alpha})^2)$ and deduce the variance using the two expectations.
    \end{proof}
    
\section{Application to the Modeling of Hair Growth}\label{sec:application to the modeling of hair growth}

    We begin this section by giving an overview of the hair growth cycle and use this to establish the assumptions we use to model a growing hair. 
    
    Moving on to modeling, we first consider a deterministic model utilizing differential equations and then move on to the stochastic model, which will be a specified resetting process. Finally, we model a haircut's impact on a growing hair and compute the marginal for our model of a hair undergoing regular haircuts.
    
    \subsection{Hair Growth Cycle and Modeling Assumptions}

    A single hair follicle cycles through 5 distinct phases in the hair growth cycle: anagen, catagen, telogen, exogen, and kenogen \cite[p.11-12]{blume2008hair}. Hair growth occurs primarily during the anagen stage, where new material is deposited within the hair shaft. Upon cessation of the anagen phase, the follicle undergoes the categen phase - a transition period between the active anagen phase and the inactive telogen phase. After the telogen phase, the hair is shed during the exogen phase, wherefore the follicle remains empty through the kenogen phase until the cycle begins anew with another anagen phase. Considered by their visual effect on hair length, these can be collected into three distinct stages of growth, rest, and absence, where hair length increases, remains at a peak, and returns to zero for a time.

    With this understanding of the hair growth cycle, we see that our initial hypothesis of constant local hair growth explained in our motivating example does not exactly align with reality. In fact, considering the cycle from the kenogen phase onwards, hair length remains at zero before beginning growth and plateaus to some constant length prior to being shed. We might choose to include each of the three stages in our model. However, the growing stage comprises approximately $80\%$ of the average 29 week cycle length \cite{saitoh1970human}. Therefore, we only consider the anagen phase (growth) and the exogen phase (shedding). See \cite{halloy2002follicular} for a continuous time discrete state space stochastic process which models the growth stages separately.
    
    Using optical microscopy and image analysis, Hayashi et al. \cite{hayashi1991measurement} observed an average scalp hair growth rate of approximately $9.5$ mm/month from adults with little to no balding. Because our goal is to predict hair length and variation over long durations, the time dependency of growth rate is of little consequence, so we assume a constant growth rate.

    Unusually, human hair - particularly scalp hair - does not undergo synchronized shedding \cite[p.42-43]{anastassakis2022androgenetic}. Therefore, we consider each hair independent of the rest. Moreover, while hair loss rate is affected by external factors such as the season \cite[p.43]{anastassakis2022androgenetic} and certainly bathing routine, because we will eventually consider a large number of hairs over many months, we assume that the distribution of shedding is identical for each successive hair grown from a single follicle. 
    
    In section \ref{sec:theory} we established a theoretical foundation for modeling stochastically resetting deterministic processes. As we explained in the introduction, the continued growth of a single hair follicle is a prime candidate to be modeled by such a process. In particular, we make the following modeling assumptions to apply our theory:\\

    We assume successive hairs growing out of a single follicle:
    \begin{enumerate}
        \item Fall out independently of other hairs,
        \item Grow at a constant rate $r>0$,
        \item Fall out with identical distribution, rate $\lambda>0$,
        \item Fall out at root only (do not break),
        \item Begin growing immediately after falling out,
        \item Fall out independently of hair length.
    \end{enumerate}

    Before we construct a stochastic model for this system, we construct a simple deterministic model to add context to the stochastic model.

    \subsection{Deterministic Hair Growth Model}\label{subsec:deterministic hair growth model}

    Here, we consider a simple deterministic model for the hair growth of an entire head of hair using differential equations. See \cite{al2012prototypic} for a more comprehensive (and scientifically grounded) deterministic model utilizing coupled differential equations.
    
    For $t\geq0$, let $X(t)$ be the average length of a head of hair at time $t\geq0$. Then, as individual hairs grow and fall out at a constant rate, we have
    \[
    X'(t)=r-\lambda X(t).
    \]
    Note that $\lambda$ being multiplied by $X(t)$ does not contradict assumption 6, but rather is due to the observation that when hairs do fall out, longer hairs will have a greater impact on the average hair length. 

    Imposing that the average hair length starts at some $X_0\geq0$, we then solve this equation through standard methods, obtaining
    \[
    X(t)=\frac{r}{\lambda}(1-e^{-\lambda t})+X_{0}e^{-\lambda t}.
    \]

    This model predicts the average length of a head of hair into the future. However, the visual appearance of a head of hair depends on more than just its average length. Crucially, the sample variance will impact the visual appearance. As an adjacent example, consider the visual differences between a mowed lawn and a natural field. Measuring the sample variance of a head of hair would be difficult and require some degree of discomfort. However, with a stochastic model, by the law of large numbers and the large number of hairs on a human head, the predicted mean and variance will give an accurate prediction of the sample mean and variance. With this goal in mind, we move on to a stochastic model.
    
    \subsection{Stochastic Hair Growth Model}\label{subsec:stochastic hair growth model}
    
    Returning to a stochastic viewpoint, let us now consider our assumptions' implications for our stochastic model. Assumption 1 allows us to model each hair follicle independently and assumptions 2 through 5 allow us to model each hair as a resetting process induced by the linear $r$-scaling function started from $0$.

    Because hair length is a linear function of the time since the last hair fell out (renewal time), assumption 6 also implies that hairs fall out independent of this renewal time. In fact this is exactly the memoryless property as stated in the proof of Theorem \ref{the:marginal of exp timing process started at a}. Exponential random variables are the only continuous random variable with the memoryless property, so assumptions 3 and 6 imply that our inter-renewal times will be $exp(\lambda)$ distributed.

    With these deductions, we define a class of resetting processes we use to model a growing hair.
    
    \begin{definition}[Growing Process]\label{def:growing process}\ \\
        Let $\lambda>0$ and $(Z_t)_{t\geq0}$ a timing process induced by inter-renewal time exp$(\lambda)$. Let $r>0$ be given and define the stochastic process $(G_t)_{t\geq0}$ by
        \[
        G_t=rZ_t.
        \]
        Call $(G_t)_{t\geq0}$ a \textbf{growing process} with \textbf{growth rate} $r$ and \textbf{renewal rate} $\lambda$.
    \end{definition}
    \begin{remark}\label{rem:growing remark}
        Because a growing process $(G_t)_{t\geq0}$ is a scaled timing process, it is a Markov process by injectivity of the scaling map and by Corollary \ref{cor:condition for markov}. Moreover, when we consider a growing process $(G_{t}^{\alpha})_{t\geq0}$ started at $\alpha\geq0$, we can easily compute its marginal law, expectation, and variance using Corollary \ref{cor:expectation and variance}, taking care to properly scale the starting point $\alpha$. In particular, the marginal is given by
        \[
        \mathbb{P}[G_{t}^{\alpha}\leq g]=
        \begin{cases} 
              0 & g\in(-\infty,0) \\
              1-e^{-\frac{\lambda g}{r}} & g\in[0,tr) \\
              1-e^{-\lambda t} & g\in[tr,tr+\alpha) \\
              1 & g\in[tr+\alpha,\infty) 
        \end{cases}.
        \]
        and the expectation by
        \[
        \mathbb{E}(G_{t}^{\alpha})=\frac{r}{\lambda}(1-e^{-\lambda t}) +\alpha e^{-\lambda t}.
        \]
        Note that the expectation of a growing process is a perfect analog to the average hair length predicted by the deterministic model.
    \end{remark}

    \subsection{Haircuts}\label{subsec:haircuts}

    Using the statements in section \ref{sec:theory}, we can now describe many characteristics of natural hair growth. In particular, considering Remark \ref{rem:mixed2} regarding the limiting behavior of a timing process, we can make strong statements regarding the limiting behavior of hair growth. However, people rarely let their hair grow for sufficient time to approximate the limiting behavior of hair growth, but rather, get haircuts before it reaches this point. To better model reality, we consider haircuts and their effect on a growing hair, introduce a hairucut operator which will act upon a growing process, and examine the marginal of the resulting process.

    Consider the following question: After an individual gets a haircut of length $\ell>0$, are all the hairs of length $\ell$? The answer is no because hairs that were initially shorter than $\ell$ were unaffected. This highlights the key haircut functionality of reducing a growing hair to the haircut length only when the hair's length is greater than the haircut length.

    We now define a haircut and let it act upon growing processes.
    
    \begin{definition}[Haircuts and haircut processes]\label{def:haircuts and haircut processes}\ \\
        For $\ell, s \geq0$, call the pair $(\ell, s )$ a \textbf{haircut} with \textbf{cuttoff length} $\ell$ and \textbf{haircut time} $ s $. Let any haircut $(\ell, s )$ act on a growing process $(G_t)_{t\geq0}$ by
        \[
        (\ell,  s ).(G_t)_{t\geq0}=((\ell,  s ).G_t)_{t\geq0}
        \]
        where
        \[
        (\ell, s ).G_t=
        \begin{cases} 
              G_t & t< s  \\
              G_{t- s }^{\min\{G_{ s },\ell\}} & t\geq  s  \\
        \end{cases}.
        \]
        By the (time homogeneous) Markov property and local linear growth of growing processes, this accurately models how a haircut acts upon a growing hair.
        We then let any sequence of haircuts $((\ell_n, s _n))_{n\in\mathbb{N}}$ act upon a growing process by inductively applying the above formulation. Call such a resultant process a \textbf{haircut process}.
    \end{definition}
    \begin{remark}\label{rem:sample path not preserved}
        Because this definition appends two growing processes together, the resulting sample path after time $ s $ will deviate from how the sample path of $(G_t)_{t\geq0}$ would look with a temporary reduction. We could have instead defined a haircut to be a random variable which reduced the growing process for a time under the right circumstances, preserving the sample path. However, the construction shown here produces an equivalent stochastic process due to the Markov property. Moreover, as we do not apply the definition too carefully, we use this more intuitive approach. The key observation is the haircut's functionality at time $ s $. Specifically, $(\ell, s ).G_ s =\min\{G_{ s },\ell\}$.
    \end{remark}

    \begin{figure}
        \centering
        \includegraphics[width=2.6666in]{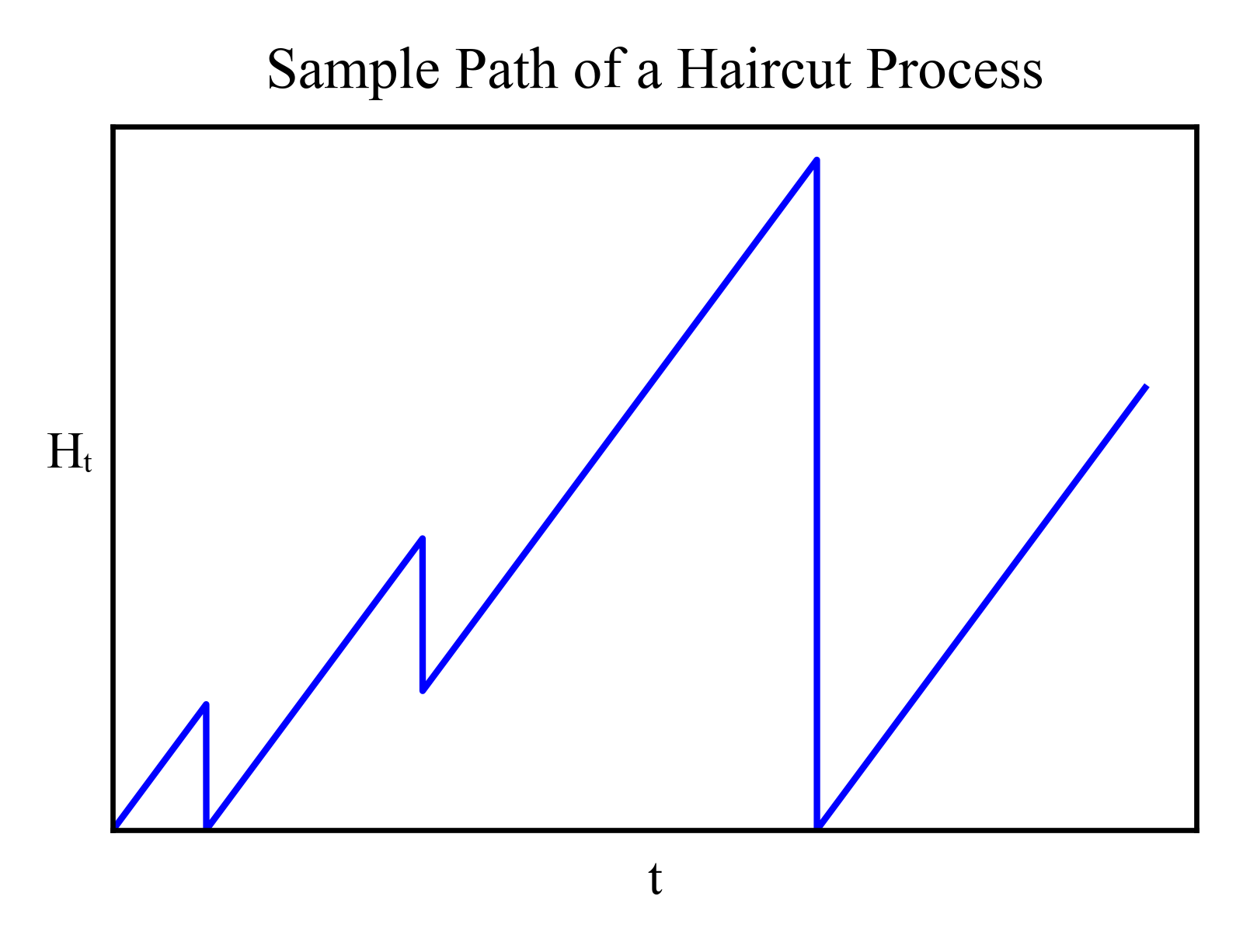}
        \caption{A sample path of a haircut process induced by a single haircut.}
        \label{fig:haircut process sample path}
    \end{figure}

    At this point in time, we have no reason to believe the marginal of a haircut process will remain tractable. Indeed, perhaps with each successive haircut the local marginal becomes more complicated, with computational complexity compounding. However, this remains to be seen. We now consider what occurs in a haircut process at haircut time - the first time the haircut process deviates from a growing process.

    \begin{proposition}[Law of haircut process at haircut time]\label{prop:cdf at haircut time}\ \\
        Let $(G_t)_{t\geq0}$ be a growing process with growth rate $r>0$ and renewal rate $\lambda>0$. Suppose $(\ell, s )$ is a haircut such that $\frac{\ell}{r}\leq  s $. Then we have that
        \[
        (\ell, s ).G_ s \sim G_{\frac{\ell}{r}}.
        \]
    \end{proposition}
    \begin{proof}
        From the definition of a haircut process, we have that
        \[
        (\ell, s ).G_ s =\min\{G_{ s },\ell\}.
        \]
        Call this random variable $M$.
        Consider that by the definition of a growing process, $G_t=rZ_t$. Thus
        \[
        \frac{1}{r}M=\min\{\frac{1}{r}G_{ s },\frac{\ell}{r}\}=\min\{Z_{ s },\frac{\ell}{r}\}
        \]
        Then, by Remark \ref{rem:minimum charactarization}, we have the alternate characterization of $Z_ s  \sim min\{exp(\lambda), s \}$. Hence,
        \[
        \frac{1}{r}M \sim \min\{exp(\lambda), s ,\frac{\ell}{r}\}.
        \]
        By assumption that $\frac{\ell}{r}\leq  s $, we see that
        \[
        \frac{1}{r}M \sim \min\{exp(\lambda),\frac{\ell}{r}\},
        \]
        which, again by Remark \ref{rem:minimum charactarization}, we understand to be distributed as $Z_\frac{\ell}{r}$. Recalling the definition of a growing process, multiplying both sides by $r$ gives the required result.
    \end{proof}
    \begin{remark}\label{rem:potent haircut}
        Returning to the definition of a timing process, as noted in the beginning of the proof of Theorem \ref{the:haircut process piecwise marginal}, at time $t\geq0$, a timing process takes values in the interval $[0,t]$. Therefore, by the definition of a growing process, for $ s \geq0$, we must have $0\leq G_ s  \leq r s $. By assuming that $\ell\leq r s $, we impose that the haircut could possibly impact the growing process.

        More contextually, a barber would be awfully confused should you ask for a haircut beyond your current hair length.

        If we assume instead that $\ell> r s $ then trivially $(\ell, s ).G_ s \sim G_{ s }$ because the growing process could never be greater than $\ell$ at time $ s $.
    \end{remark}

    This symmetry is quite surprising, arising as a result of the coincidental compatibility of the two minimums in both the alternative characterization of a timing process explained in Remark \ref{rem:minimum charactarization} and the haircut functionality explained in Remark \ref{rem:sample path not preserved}.
    
    Let us examine the implications of this result. Observe that the computed distribution $G_{\frac{\ell}{r}}$ is independent of $ s $. This means that under our model, so long as the haircut could possibly have cut a given hair, the distribution of hair length at haircut time is fixed.

    This alone is interesting, but using this proposition, we can make a stronger statement. With the resultant distribution in mind, return now to the definition of a haircut process, and consider what we can say about the haircut process's marginal law after the single haircut occurs. In particular, consider
    \[
    (G_{t}^{\text{Min}\{G_{ s },\ell\}})_{t\geq0} \sim (G_{t}^{G_{\frac{\ell}{r}}})_{t\geq0}.
    \]
    Following on from Definition \ref{def:timing starting at alpha}, for $\alpha\geq0$, $(G_{t}^{\alpha})_{t\geq0}$ is a growing process started at $\alpha$. Therefore, we have that $(G_{t}^{G_{\frac{\ell}{r}}})_{t\geq0}$ is a growing process started at $G_{\frac{\ell}{r}}$. In other words, the continued behavior of a growing process first observed at time $\frac{\ell}{r}$. Hence,
    \[
    (G_{t}^{\text{Min}\{G_{ s },\ell\}})_{t\geq0} \sim (G_{\frac{\ell}{r}+t})_{t\geq0}.
    \]
    While this statement is describing an equivalence of stochastic processes, it is important to note that one cannot sample two growing processes and append them together to construct a haircut process as that would ignore the dependence of the random variables in a haircut process on the boundary of the haircut time. However, the above observation is sufficient to describe the entire marginal of $(\ell,  s ).(G_t)_{t\geq0}$.

    Note that we have only been considering a haircut process induced by a single haircut. However, because the law of a haircut process after the first haircut is identical to the tail of another growing process, the argument shown above can be applied inductively to any sequence of successive haircuts, allowing us to describe the piecewise law of an arbitrary haircut process. In particular, we can describe an arbitrary haircut process's marginal.

    Intuitively, one might expect a haircut to impose some unnatural order on the hair. However, Proposition \ref{prop:cdf at haircut time} and the arguments above tell us that cutting a hair can only move a hair's growth to a previous point in its natural growth cycle. Any imposed order would be between separate hairs rather than upon a single hair.
    
    For use in Section \ref{sec:quantifying unkempt hair}, we now describe the case of regular haircuts with constant cuttoff length.

    \begin{theorem}[Haircut process piecewise law]
    \label{the:haircut process piecwise marginal}\ \\
        Let $(G_t)_{t\geq0}$ be a growing process with growth rate $r>0$. Suppose the sequence of consecutive haircuts $(\ell, s +ni)_{n\in\mathbb{N}_0}$, with $r s \geq\ell$ and period $i>0$, acts on the growing process to induce the haircut process $(H_t)_{t\geq0}$. Then for $n\in\mathbb{N}$ the local law between the $n$th and $(n+1)$th haircuts is given by
        \[
        (H_t)_{t\in[ s +ni, s +(n+1)i)}\sim (G_t)_{t\in[\frac{\ell}{r},\frac{\ell}{r}+i)}.
        \]
    \end{theorem}

\section{Computing Your Ideal Haircut}\label{sec:quantifying unkempt hair}

\subsection{Quantifying Unkempt Hair}
    Now, having developed a model for a growing hair receiving regular haircuts, we consider quantifying the appearance of a head of hair. This section seeks to answer the question: what is my ideal haircut routine? Given the subjective nature of the question, the application of the mathematical results in this section is similarly subjective. We include this section so the reader may conclude with a novel exercise, answering for themselves what their ideal haircut routine is.

    First, we must establish a way of quantifying the appearance of an individual's head of hair.
    
    \begin{definition}[Rattiness]\label{def:rattiness}\ \\
        Let $K$ be a random variable defined on probability space $(\Omega,\mathcal{F},\mathbb{P})$. Assume $K$ has existing non zero expectation and finite variance. Then, define the \textbf{rattiness} of K to be
        \[
        R(K)=\frac{\sqrt{\text{VAR}(K)}}{\mathbb{E}(K)}.
        \]
    \end{definition}
    \begin{remark}\label{rem:coefficient of variation}
        This formulation is classically called the coefficient of variation of $K$.
    \end{remark}

    The thinking here is that the sample variation in an individual's hair lengths influences the appearance. A constant variation would be more impactfull with a shorter mean. Hence dividing by the expectation. Then, to put the variance on the same scale as the expectation, we take its square root.

    It is important to note that the goal of this function is not to declare a best hair appearance, but to provide a standardized scale on which an individual can describe their preference.

    While the rattiness of a random variable is purely a function of expectations, due to the law of large numbers and the large number of hairs on the average human head, the sample rattiness will be nearly identical to the expected rattiness.

    We first compute the rattiness of the random variables in a growing process.
    
    \begin{proposition}[Rattiness of Growing Process]\label{prop:rattiness of growing process}\ \\
        Let $(G_t)_{t\geq0}$ be a growing process with growth rate $r>0$ and renewal rate $\lambda>0$. Then for $t>0$,
        \[
        R(G_t)=g(\lambda t)
        \]
        where $g:[0,\infty)\xrightarrow{}[0,1)$, defined by
        \[
        g(x)=
        \begin{cases} 
                  0 & x=0 \\
                  \frac{\sqrt{1-2xe^{-x}-e^{-2x}}}{1-e^{-x}} & x>0
        \end{cases},
        \]
        is a strictly increasing, continuous bijection, with $\lim_{x\to\infty} g(x)=1$.
    \end{proposition}
    \begin{proof}
        Note that the rattiness of a random variable is invariant under (non-zero) scaling. Therefore, as growing processes are scaled timing processes, we need only compute the rattiness of a timing process $(Z_t)_{t\geq0}$.

        Let $t>0$ be given. Then, by substituting $\alpha=0$ into the expectation and variance computed in the Corollary \ref{cor:expectation and variance}, we obtain
        \[
        \mathbb{E}(Z_{t})=\frac{1}{\lambda}(1-e^{-\lambda t})
        \]
        \[
        Var(Z_{t})=\frac{1}{\lambda^2}(1-2t\lambda e^{-t\lambda}-e^{-2t\lambda})
        \]
        both finite and non-zero. Therefore, by the definition of rattiness, we compute
        \[
        R(G_t)=R(Z_t)=\frac{\sqrt{1-2\lambda te^{-\lambda t}-e^{-2\lambda t}}}{1-e^{-\lambda t}}=g(\lambda t).
        \]
        We leave it to the reader to verify the properties of g.
    \end{proof}

    In section \ref{sec:application to the modeling of hair growth} and in particular Theorem \ref{the:haircut process piecwise marginal}, we computed the marginal of haircut processes, finding them to be repeated sections of a growing process's marginal. Using these observations and the rattiness of a growing process, we find that the rattiness of any haircut process is simply repeated sections of the rattiness of a growing process. Since the rattiness of a growing process is independent of growth rate $r>0$ and maps nicely onto $(0,1)$, we take this to be our scale on which to describe the appearance of a head of hair.

    For use in the final subsection, we explicitly compute the rattiness of a haircut process receiving regular haircuts.
    
    \begin{theorem}[Rattiness of Haircut Process]
    \label{the:rattiness of haircut process}\ \\
        Let $(H_t)_{t\geq0}$ be the haircut process constructed in Theorem \ref{the:haircut process piecwise marginal}. Then, for $t\geq0$, the rattiness of $H_t$ is given by
        \[
        R(H_t)=
        \begin{cases} 
                  g(\lambda t) & t\in[0, s ) \\
                  g(\lambda(\frac{\ell}{r}+(t-( s +ni)))) & t\in[ s +ni, s +(n+1)i) \text{  for some $n\in\mathbb{N}$}.
        \end{cases}
        \]
    \end{theorem}

\subsection{Your Ideal Haircut}
    
    Finally, equipped with a standardized 0-1 scale on which to quantify hair unkemptness, we prescribe a method for determining an individual's ideal haircut routine.

    Suppose an individual has hair growth rate $r>0$ and loses hair at a rate of $\lambda>0$. Then, suppose they are comfortable with hair rattiness in the interval $[a,b]$ where $a,b\in[0,1]$ with $a<b$. We compute optimal the cuttoff length $\ell\geq0$ and interval $i>0$ between haircuts.

    Assuming an ideal haircut routine is frugal, we wish to maximize the time spent in this region of comfort to minimize the number of haircuts. Then, because $g$ is strictly increasing and continuous, we impose that the rattiness at haircut time should be $a$ and rattiness immediately before haircut time should be $b$. Deferring to Theorem \ref{the:rattiness of haircut process}, we have that the rattiness at haircut time is $g(\lambda\frac{\ell}{r})$ and rattiness immediately before haircut time is $g(\lambda(\frac{\ell}{r}+i))$. Because g is a bijection, we can solve
    \[
    g(\lambda\frac{\ell}{r})=a \implies \ell=\frac{r}{\lambda}g^{-1}(a)
    \]
    \[
    g(\lambda(\frac{\ell}{r}+i))=b \implies i=\frac{1}{\lambda}g^{-1}(b) - \frac{\ell}{r}=\frac{1}{\lambda}(g^{-1}(b)-g^{-1}(a)).
    \]
    Recalling that the hair cycle lasts an estimated $29$ weeks \cite{saitoh1970human}, we compute $\lambda \approx 0.15$ /month. Moreover, an estimated average growth rate is $r \approx 9.5$ mm/month \cite{hayashi1991measurement}. With both estimates given for scalp hair in an adult with little to no balding, the ideal haircut length and interval can now be computed.

    As one might expect, the ideal haircut length $\ell$ is directly proportional to growth rate $r$. Perhaps more interesting, the ideal haircut interval $i$ is independent of growth rate $r$ and inversely proportional to loss rate $\lambda$, meaning that losing more hair is better for ones wallet, while growing more hair is inconsequential. The bald and balding are the real winners it seems.

    The astute reader may have noticed that, implicit in the use of the law of large numbers, this formulation models each hair as identically distributed and in particular being acted upon by haircuts of constant cuttoff length. In reality, hair growth rate will vary between follicles and follicle location. Moreover, as the head is dome shaped, a single physical haircut requires different cuttoff lengths for each haircut acting upon the affected follicle, as the hairs hang from the root. Hence, the model cannot describe style or hairdo. However, we figure that an individual willing to let a calculation decide their haircut is not looking for fashion advice in a mathematical journal.

\section{Other Modeling Exercises}\label{sec:other exercises}
    We leave the reader with a few exercises that can be modeled using the theory in this article:

    \begin{enumerate}
        \item \textbf{Galactic Conquest.}\\
        Suppose there exists an alien race which is determined to dominate the galaxy, and that during times of peace, the aliens reproduce exponentially with time. In their quest, they engage in galactic wars according to the arrivals of a Poisson process. But, as the aliens are poor fighters, with each war, they are forced back to their home planet, where their population is reduced to a set amount. What degree of pacifism is required to take over the galaxy? (For what parameters of the Poisson process does the expected population diverge to infinity?)
        
        \item \textbf{The Candy Jar.}\\
        A receptionist offers candy from a jar on their desk. They refill the jar to the top according to the arrivals of a Poisson process. Suppose that, when the jar is non-empty, candy is taken at a constant rate. Assuming the receptionist has done this for a long while, for what proportion of time is the jar empty? Suppose the employees are rude, and complain as soon as the jar is empty. After the receptionist refills the jar, how long can they expect to wait before receiving a complaint?

        \item \textbf{The Automatic Door}.\\
        An external entrance is equipped with an automatic door and people walk through it according the the arrivals of a Poisson process. When a person approaches the door it quickly opens if closed, or remains open if already open. If no person walks through in a fixed amount of time, the door closes quickly. When the door is open, the building incurs a fixed heating cost per unit time, and the door incurs a fixed electricity and maintenance cost each time it opens and closes. For what combinations of parameters is it cheaper to leave the door open than to use the automatic system?
    \end{enumerate}

    For the second and third questions, you may need to deduce a mean stopping time. This can be done using the infinitesimal generator of the process described in Theorem \ref{the:marginal of exp timing process started at a}, which is given by
    \[
    \mathbb{L}_{Z}f(x)=\lim_{t\downarrow0}(\frac{1}{t}(\mathbb{E}(f(Z_t^x))-f(x)))=f'(x)-\lambda(f(x)-f(0))
    \]
    for any differentiable $f:\mathbb{R}\to\mathbb{R}$ and $x\geq0$. This can be deduced by conditioning upon whether or not the process is reset before time $t$. Alternatively, having completed the first exercise, you may be able to spot a relevant Martingale.
    
\bibliographystyle{vancouver}
\bibliography{Sources}
    
\end{document}